\documentclass[a4paper,11pt]{article}

\usepackage[T1]{fontenc}
\usepackage{lmodern}

\usepackage{dsfont}

\usepackage[latin1]{inputenc}
\usepackage{amsmath}
\usepackage{amsthm}
\usepackage{amssymb}
\usepackage{mathrsfs}
\usepackage{graphicx}
\usepackage[all]{xy}
\usepackage{hyperref}

\usepackage{stmaryrd}
\usepackage{caption}

\usepackage{abstract} 

\newtheorem{thm}{Theorem}[section]
\newtheorem{cor}[thm]{Corollary}
\newtheorem{claim}[thm]{Claim}
\newtheorem{fact}[thm]{Fact}

\newtheorem{lemma}[thm]{Lemma}
\newtheorem{prop}[thm]{Proposition}

\theoremstyle{definition}
\newtheorem{definition}[thm]{Definition}
\newtheorem{ex}[thm]{Example}
\newtheorem{remark}[thm]{Remark}
\newtheorem{question}[thm]{Question}

\def\rquotient#1#2{%
	\makeatletter
	\raise.3ex\hbox{$#1$}/\lower.3ex\hbox{$#2$}%
	\makeatother
}	

\makeatletter
\newcommand{\subjclass}[2][2010]{%
	\let\@oldtitle\@title%
	\gdef\@title{\@oldtitle\footnotetext{#1 \emph{Mathematics subject classification.} #2}}%
}
\newcommand{\keywords}[1]{%
	\let\@@oldtitle\@title%
	\gdef\@title{\@@oldtitle\footnotetext{\emph{Key words and phrases.} #1.}}%
}
\makeatother

\newcommand{\Address}{{
		\bigskip
		\small
		
		\textsc{D\'epartement de Math\'ematiques B\^atiment 307, Facult\'e des Sciences d'Orsay, Universit\'e Paris-Sud, F-91405 Orsay Cedex, France.}\par\nopagebreak
		\textit{E-mail address}: \texttt{anthony.genevois@math.u-psud.fr}
		
		\medskip
		
		\textsc{D\'epartement de Math\'ematiques, Facult\'e des Sciences, Aix-Marseille Universit\'e, 3 place Victor Hugo, 13331 Marseille cedex 3, France.}\par\nopagebreak
		\textit{E-mail address}: \texttt{arnaud.stocker@univ-amu.fr}
		
}}

\title{Partially CAT(-1) groups are acylindrically hyperbolic}
\date{\today}
\author{Anthony Genevois and Arnaud Stocker}


\begin{document}

\maketitle

\begin{abstract}
In this paper, we show that, if a group $G$ acts geometrically on a geodesically complete CAT(0) space $X$ which contains at least one point with a CAT(-1) neighborhood, then $G$ must be either virtually cyclic or acylindrically hyperbolic. As a consequence, the fundamental group of a compact Riemannian manifold whose sectional curvature is nonpositive everywhere and negative in at least one point is either virtually cyclic or acylindrically hyperbolic. This statement provides a precise interpretation of an idea expressed by Gromov in his paper \emph{Asymptotic invariants of infinite groups}.

\medskip
Dans cet article, nous d\'emontrons que, si un groupe $G$ agit g\'eom\'etriquement sur un espace CAT(0) $X$ qui est g\'eod\'esiquement complet et qui contient au moins un point admettant un voisinage CAT(-1), alors $G$ doit \^etre ou bien acylindriquement hyperbolique ou bien virtuellement cyclique. Par cons\'equent, le groupe fondamental d'une vari\'et\'e riemannienne compacte dont la courbure sectionnelle est n\'egative ou nulle partout et strictement n\'egative en au moins un point doit \^etre acylindriquement hyperbolique ou virtuellement cyclique. Cet \'enonc\'e propose une interpr\'etation pr\'ecise et moderne d'une id\'ee de Gromov d\'ecrite dans \emph{Asymptotic invariants of infinite groups}.  
\end{abstract}

\tableofcontents

\section{Introduction}

A key concept in geometric group theory is the notion of ``curvature''. Usually, exhibiting some negative curvature in the geometry of a given group provides interesting information on the algebraic properties of our group. The first instance of such a geometry in group theory was \emph{small cancellation} (see \cite{LS}, or \cite{McCammondWise} for a more geometric approach), next generalised by Gromov with the seminal concept of \emph{hyperbolic groups} \cite{GromovHyp}. There, Gromov also suggests the definition of \emph{relatively hyperbolic groups}, allowing some non hyperbolic subspaces but concentrating them into a controlled collection of subgroups. We refer to the survey \cite{HruskaRH} and references therein for more information on relatively hyperbolic groups and their historical development. More recently, Osin introduced \emph{acylindrically hyperbolic groups} \cite{OsinAcyl} in order to unify several classes of groups considered as ``negatively-curved''. 

Although acylindrically hyperbolic groups generalise relatively hyperbolic groups, many arguments used in the context of relatively hyperbolic groups turn out to hold in the context of acylindrically hyperbolic groups as well \cite{DGO}. One of the most impressive algebraic consequence is that acylindrically hyperbolic groups are SQ-universal. The first motivating examples of such groups were mapping class groups \cite{BBboundedcohoMCG} and outer automorphism groups of free groups \cite{BestvinaFeighnOutFnHyp}. However, since then a large amount of articles have been dedicated to the recognition of acylindrically hyperbolic groups among familiar classes of groups, including graph products \cite{MinasyanOsin}, 3-manifold groups \cite{MinasyanOsin}, Cremona groups \cite{LonjouCremonaAcylHyp}, small cancellation groups \cite{C(7), CubeSmallCancelAcylHyp, coningoff}, groups of deficiency at least two \cite{DeficiencyAcylHyp}, Artin groups \cite{CalvezWiestArtinAcylHyp, NoteAcylHyp}, diagram groups \cite{article3} and graph braid groups \cite{MoiGraphBraid}. So far, acylindrical hyperbolicity is the most general convincing notion of groups admitting some ``negatively-curved behaviour''.  

In this paper, we are interested in some ``nonpositively-curved'' groups, namely CAT(0) groups, ie., groups acting geometrically on CAT(0) spaces. These spaces generalise Riemannian manifolds whose curvature is everywhere nonpositive, also known as Hadamard manifolds. Loosely speaking, what we show is that a CAT(0) group which is ``partially hyperbolic'' turns out to be acylindrically hyperbolic. More precisely, the main result of the article is:

\begin{thm}\label{thm:mainintro}
Let $G$ be a group acting geometrically on some geodesically complete proper CAT(0) space $X$. Suppose that $X$ contains at least one point with a CAT(-1) neighborhood. Then $G$ must be either virtually cyclic or acylindrically hyperbolic.
\end{thm}

\noindent
As a consequence of this criterion, one immediately gets the following statement:

\begin{cor}
The fundamental group of a compact Riemannian manifold whose sectional curvature is nonpositive everywhere and negative in at least one point is either virtually cyclic or acylindrically hyperbolic.
\end{cor}

In fact, this statement has motivated the present work, following a remark made by Gromov in his article \emph{Asymptotic invariants of infinite groups} \cite{GromovAsymptotic}.

\begin{quotation}
``In fact, the fundamental groups $\Gamma$ of manifolds $V$ with $K \leq 0$, such that $K<0$ at some point $v \in V$ \emph{generalize} in a certain way the hyperbolic groups''. (\cite{GromovAsymptotic}, page 147.)
\end{quotation}

Our strategy to prove Theorem \ref{thm:mainintro} is the following. The first step is to construct a geodesic ray in our CAT(0) space passing periodically through CAT(-1) points (see Proposition \ref{prop1} below). Next, we show that such a ray must be a \emph{contracting} (or equivalently, \emph{rank-one}) geodesic ray in our CAT(0) space. More precisely, we prove the following general criterion:

\begin{prop}[see Proposition \ref{prop2} below]
Let $X$ be a CAT(0) space and $\gamma : \mathbb{R} \to X$ a geodesic line passing through infinitely many CAT(-1) points in such a way that the distance between any two consecutive such points is uniformly bounded from above and from zero and that the radii of the CAT(-1) neighborhoods are uniformly bounded away from zero. Then $\gamma$ is a contracting geodesic. 
\end{prop}

The existence of such a ray implies the existence of a \emph{contracting} (or equivalently, \emph{rank-one}) isometry in our group, which finally implies that our group must be either virtually cyclic or acylindrically hyperbolic. 

It is worth noticing that, since our strategy is to construct a rank-one isometry, Theorem \ref{thm:mainintro} is related to the famous

\medskip \noindent
\textbf{Rank Rigidity Conjecture \cite{BallmannBuyalo}.} \emph{Let $X$ be a proper geodesically complete CAT(0) space and $G$ an infinite group acting geometrically on $X$. If $X$ is irreducible, then $X$ is a higher rank symmetric space, or a Euclidean building of dimension at least two, or $G$ contains a rank-one isometry.}

\medskip
In one direction, our theorem is a consequence of the conjecture. And in the oppositive direction, our theorem may be applied to verify the conjecture in particular cases. However, we do not know examples where our criterion applies and where the Rank Rigidity Conjecture remains unknown.

The article is organised as follows. In Section \ref{section:acyl}, we recall briefly several known characterisations of acylindrical hyperbolicity among CAT(0) groups, including the criterion which we will use to prove Theorem \ref{thm:mainintro} in Section \ref{section:mainthm}. Finally, in Section \ref{section:nonRH}, we end with two concluding remarks. It includes an example of a partially CAT(-1) group which is not relatively hyperbolic, showing that the conclusion of Theorem \ref{thm:mainintro} cannot be strengthened to get some relative hyperbolicity.

\paragraph{Acknowledgments.} The authors are grateful to their advisor, Peter Ha\"{i}ssinsky, and to Alexander Lytchak, for his useful comments on an earlier version of our article, which lead to a simplification of the proof of Proposition \ref{prop2}.

\section{Preliminaries}

\paragraph{CAT($\kappa$) spaces.}
In this paragraph, we briefly recall the definition of CAT($\kappa$) spaces for $\kappa\leq 0$ and the notion of angle in these spaces which will be used in the paper; we refer to \cite{MR1744486} for more information and proofs of the statements of this section. For $\kappa\leq 0$ let $(M_\kappa^2,d_\kappa)$ be the simply connected Riemannian 2-manifold of constant curvature equal to $\kappa$. That is

$$
M_\kappa^2 = \left\{
\begin{array}{ll}
	\frac{1}{\sqrt{-\kappa}}\mathbb{H}^2 & \mbox{if } \kappa<0\\
	\mathbb{E}^2 & \mbox{if } \kappa=0.
\end{array}
\right.
$$
 A \emph{triangle} in a geodesic metric space $X$ consists of three points $p,q,r\in X$ (the vertices) together with a choice of geodesics segments $[p,q]$, $[q,r]$ and $[r,p]$. Such a triangle will be denoted $\Delta(p,q,r)$ (this notation is not accurate because the geodesic segments are, in general, not unique). A \emph{comparison triangle} in $M_\kappa^2$ for $\Delta(p,q,r)$ is a triangle $\Delta(\bar{p},\bar{q},\bar{r})$ in $M_\kappa^2$ such that $d_\kappa(\bar{p},\bar{q})=d(p,q)$, $d_\kappa(\bar{q},\bar{r})=d(q,r)$ and $d_\kappa(\bar{r},\bar{p})=d(r,p)$. Such a triangle always exists and is unique up to isometry. A point $\bar{x}\in[\bar{p},\bar{q}]$ is called a \emph{comparison point} for $x\in[p,q]$ if $d(p,x)=d_\kappa(\bar{p},\bar{x})$. Similarly, we can define comparison points on $[\bar{q},\bar{r}]$ and $[\bar{p},\bar{r}]$.

\begin{definition}
Let $(X,d)$ be a geodesic metric space and $\kappa\leq 0$.
\begin{itemize}
\item A triangle $\Delta(p,q,r)$ in $X$ satisfies the \emph{CAT($\kappa$) inequality} if for all $x,y\in\Delta(p,q,r)$ and all comparison points $\bar{x}$, $\bar{y}\in\Delta(\bar{p},\bar{q},\bar{r})$,
$$d(x,y)\leq d_\kappa(\bar{x},\bar{y}).$$
\item $X$ is a \emph{CAT($\kappa$) space} if all triangles satisfies the CAT($\kappa$) inequality. 
\end{itemize}
\end{definition}

Basic examples of CAT($\kappa$) spaces, which motivated the above definition, are the simply connected Riemannian manifolds of sectional curvature $\leq \kappa$.

In CAT($\kappa$) spaces, $\kappa\leq 0$, one can define a notion of angle:

\begin{definition}
Let $(X,d)$ be a CAT($\kappa$) space, $\kappa\leq 0$ and $p$, $x$, $y$ three distinct points in $X$. Let $c_x$, $c_y$ be geodesic paths from $p$ to $x$ and $y$ respectively. Then the limit
$$\lim\limits_{t\rightarrow 0}\angle_{p}^{(0)}(c_x(t),c_y(t)),$$
exists, where $\angle_{p}^{(0)}(c_x(t),c_y(t))$ is the angle at $\bar{p}$ in a comparison triangle for $\Delta(p,c_x(t),c_y(t))$ in $M_0^2$. This limit is called the \emph{Alexandrov angle $\angle_p(x,y)$ at $p$ between $x$ and $y$}.
\end{definition}

The Alexandrov angle $\angle_p(x,y)$ is always smaller than the angle $\angle_{p}^{\kappa}(x,y)$ at $\bar{p}$ in a comparison triangle $\Delta(\bar{p},\bar{x},\bar{y})\subset M_\kappa^2$ and satisfies the triangular inequality.

\paragraph{Acylindrically hyperbolic CAT(0) groups.}\label{section:acyl}
We do not define acylindrically hyperbolic groups here, and refer to \cite{OsinAcyl} and references therein for more information on this class of groups. We only mention the following characterisation of acylindrical hyperbolicity in the context of CAT(0) groups. (This characterisation is a consequence of the combination of results proved in \cite{Sistohypembed, arXiv:1112.2666, MR3339446, BallmannBuyalo}; see \cite[Theorem 6.40]{MoiHypCube} for more details).

\begin{thm}\label{thm:whenacylhyp}
Let $G$ be a non virtually cyclic group acting geometrically on a CAT(0) space $X$. Then $G$ is acylindrically hyperbolic if and only if $X$ contains a contracting geodesic ray.
\end{thm}

\noindent
Recall that, given a metric space $(S,d)$, a subspace $Y \subset S$ is \emph{contracting} if there exists some $B \geq 0$ such that the nearest-point projection onto $Y$ of any ball which is disjoint from $Y$ has diameter at most $B$. The strategy to prove our main theorem will be to construct a contracting geodesic line.

\section{Partially CAT(-1) groups}\label{section:mainthm}

\noindent
In this section, we prove the main result of our article. Namely:

\begin{thm}\label{thm:main}
Let $G$ be a group acting geometrically on some geodesically complete proper CAT(0) space $X$. Suppose that $X$ contains at least one point with a CAT(-1) neighborhood. Then $G$ must be either virtually cyclic or acylindrically hyperbolic.
\end{thm}

\noindent
For convenience, from now on we will refer to points admitting CAT(-1) neighborhoods as \emph{CAT(-1) points}. Our theorem will be a consequence of Propositions \ref{prop1} and \ref{prop2} below. The first proposition shows that $X$ must contain some geodesic line passing through infinitely many CAT(-1) points in such a way that the distance between any two consecutive such points is uniformly bounded. Next, our second proposition states that such a line must be slim, or equivalently, contracting. The conclusion finally follows from Theorem \ref{thm:whenacylhyp}.

\begin{prop}\label{prop1}
Let $X$ be a geodesically complete, cocompact and proper CAT(0) space which contains at least one point with a CAT(-1) neighborhood. Then there exists some geodesic ray $r : [0,+ \infty) \to X$ and an increasing sequence $(t_n)$ of nonnegative reals tending to $+ \infty$ such that for every $n \geq 0$, $r(t_n)$ has a CAT(-1) neighborhood with radii uniformly bounded away from zero  and such that the auxiliary sequence $(t_{n+1}-t_n)$ is bounded. 
\end{prop}

\noindent
The proposition will follow from the following two lemmas:

\begin{lemma}
Let $X$ be a CAT(0) space and $R,\delta,\epsilon>0$ some fixed constants. Set $k=\frac{\delta}{\epsilon}+1$. For every geodesic ray $r$, the segment $[r(0),x]$ intersects $B(r(t),\epsilon)$ for any $t\in[0,\delta]$ and $x\in B(r(t+kR),R)$.
\end{lemma}

\begin{proof}
We note $o=r(0)$, $z=r(t)$, $y=r(t+kR)$. Consider $x\in B(y,R)$ and a comparison triangle $\Delta(\bar{o},\bar{y},\bar{x})$ in $\mathbb{E}^2$ for the geodesic triangle $\Delta(o,y,x)$. Let $\bar{\omega}$ be the intersection between $[\bar{o},\bar{x}]$ and the parallel line to $(\bar{y},\bar{x})$ passing through $\bar{z}$ and let $\omega$ be the corresponding point on $[0,x]$. Then, Thales theorem gives
\[ \frac{d(\bar{o},\bar{z})}{d(\bar{o},\bar{y})} = \frac{d(\bar{\omega},\bar{z})}{d(\bar{x},\bar{y})},  \]
and with the CAT($0$) inequality we get
\begin{align*}
d(\omega,z)\leq d(\bar{\omega},\bar{z})=\frac{d(\bar{o},\bar{z})}{d(\bar{o},\bar{y})}d(\bar{x},\bar{y})\leq\frac{\delta R}{t+kR}<\epsilon.
\end{align*}
\end{proof}

\begin{lemma}\label{lemma2}
Let $X$ be a cocompact proper CAT(0) space which contains at least one point with a CAT(-1) neighborhood and $R,\epsilon,c>0$ some fixed constants. There exists some $k=k(R,\epsilon,c)$ such that, if $o,z,y,x\in X$ are points satisfying:
\begin{enumerate}
	\item $z$ belongs to $[z,y]$; 
	\item $d(x,y)\leq R$ and $d(z,y)\geq k$;
	\item $B(z,c)$ is a CAT(-1) neighborhood of $z$;
\end{enumerate}
then the segment $[o,x]$ intersects the ball $B(z,\epsilon)$.
\end{lemma}

\begin{proof} 
By contradiction, assume that, for all $n\in\mathbb{N}$, there exist $o_n,x_n,y_n,z_n \in X$ such that:
	\begin{itemize}
		\item $z_n$ belongs to $[z_n,y_n]$; 
		\item $d(x_n,y_n)\leq R$ and $d(z_n,y_n)\geq n$;
		\item $B(z_n,c)$ is a CAT(-1) neighborhood of $z_n$;
		\item $[o_n,x_n]\cap B(z_n,\epsilon)=\emptyset$.
	\end{itemize} 
This last assumption, together with the previous lemma, implies that 
$$n \leq d(z_n,y_n) \leq \left( \frac{d(o_n,z_n)}{\varepsilon}+1 \right) \cdot R$$ 
so $d(o_n,z_n)\xrightarrow[n\rightarrow+\infty]{}+\infty$. Since $X$ is cocompact, up to translation, we may assume that the sequence $(z_n)$ stays in a compact $K \subset X$; and since $d(o_n,z_n)$ and $d(z_n,y_n)$ both tend to $+\infty$, we deduce from Arzel\`a-Ascoli theorem that, up to subsequence, the segments $[o_n,y_n]$ tend to a geodesic line $l$. Similarly, because $[o_n,x_n]$ is included into the $R$-neighborhood of $[o_n,y_n]$ as a consequence of the convexity of the distance, up to a subsequence the segments $[o_n,x_n]$ converge to a geodesic line $l'$ parallel to $l$. These lines are distinct since $d(z_n,[o_n,x_n])\geq\epsilon$ for all $n$. The contradiction comes from the Flat Strip Theorem \cite[Theorem II.2.13]{MR1744486} and the following fact, stating that a limit of CAT(-1) points must be CAT(-1) as well:
	
	\begin{fact}\label{fact}
	Let $(z_n)$ be a sequence of points converging to some $z \in X$ such that $B(z_n,c)$ is CAT(-1) for all $n$. Then $B(z,\frac{c}{2})$ is a CAT(-1) neighborhood of $z$.
\end{fact}

\noindent
Indeed, for $n$ large enough, $B(z,\frac{c}{2})\subset B(z_n,c)$. The conclusion follows from the convexity of balls.
\end{proof}

\noindent
Now we are ready to prove Proposition \ref{prop1}. For $o,z\in X$ and $R>0$ we define the \emph{shadow} $\mathcal{O}_o(z,R)$ to be the set of $y\in X$ such that the segment $[o,y]$ meets the ball centered at $z$ of radius $R$.

\begin{proof}[Proof of Proposition \ref{prop1}]
Let $Y=X\slash G$ be a compact quotient of $X$ where $G<\mathrm{Isom}(X)$ and $z\in X$ a point such that $B(z,2\epsilon)$ is a CAT(-1) neighborhood for some $\epsilon>0$. Choose $R\geq 2\epsilon+\mathrm{diam}(Y)$ and $K\geq\max(R+2\epsilon,k(R,\epsilon,2\epsilon))$ where $k(R,\epsilon,2\epsilon)$ is the constant from Lemma \ref{lemma2}. Finally, we fix a base point $o\in X$.

\medskip \noindent
We now construct by induction a sequence of points $(z_n)$ in the $G$-orbit of $z$ such that for all $n \geq 1$:
\begin{enumerate}
	\item $\mathcal{O}_o(z_n,\epsilon)\subset\dots\subset\mathcal{O}_o(z_1,\epsilon)$;
	\item there exist $z_n^1, \ldots, z_n^n \in [o,z_n]$ such that $z_n^k \in B(z_k,\epsilon)$, $\epsilon\leq d(z_n^k,z_n^{k+1})\leq K+R+\epsilon$ and $z_n^n=z_n$.
\end{enumerate}
\textbf{Step $n=1$}. Simply take $z=z_1$.

\medskip \noindent
\textbf{From step $n$ to step $n+1$}. Assume that $z_1,\dots,z_n$ are constructed. Since $X$ is geodesically complete, we can extend the segment $[o,z_n]$ to a geodesic ray $r_n$ and consider the point $y_n$ on $r_n$ such that $d(y_n,z_n)=K$. By the choice of $K$ and Lemma \ref{lemma2}, we have
\[  \mathcal{O}_{o}(y_n,R) \subset\mathcal{O}_o(z_n,\epsilon).  \]
Moreover, since the orbit of $z$ is $\mathrm{diam}(Y)$-dense, by choice of $R$, there is some $z_{n+1}\in G \cdot z$ such that $B(z_{n+1},\epsilon)\subset B(y_n,R)$. It follows that 
\[ \mathcal{O}_o(z_{n+1},\epsilon)\subset\mathcal{O}_o(y_n,R)\subset\mathcal{O}_o(z_n,\epsilon),\]
and so, $[o,z_{n+1}]$ meets every ball $B(z_k,\epsilon)$, $k=1,\dots,n$. Finally, if one fixes some $z_{n+1}^{k}\in B(z_k,\epsilon)\cap[o,z_{n+1}]$, then for $k\leq n-1$ we have
\[ d(z_{n+1}^k,z_{n+1}^{k+1})\leq d(z_{n+1}^k,z_k) +d(z_k,y_{k+1})+d(y_{k+1},z_{n+1}^{k+1})\leq \epsilon+K+R, \]
and
\[ d(z_{n+1}^k,z_{n+1}^{k+1})\geq d(z_k,y_{k+1})- d(z_{n+1}^k,z_k)-d(y_{k+1},z_{n+1}^{k+1})\geq K-R-\epsilon\geq \epsilon. \]
\begin{figure}
\begin{center}
\includegraphics[trim={0 0 0 0},clip,scale=0.35]{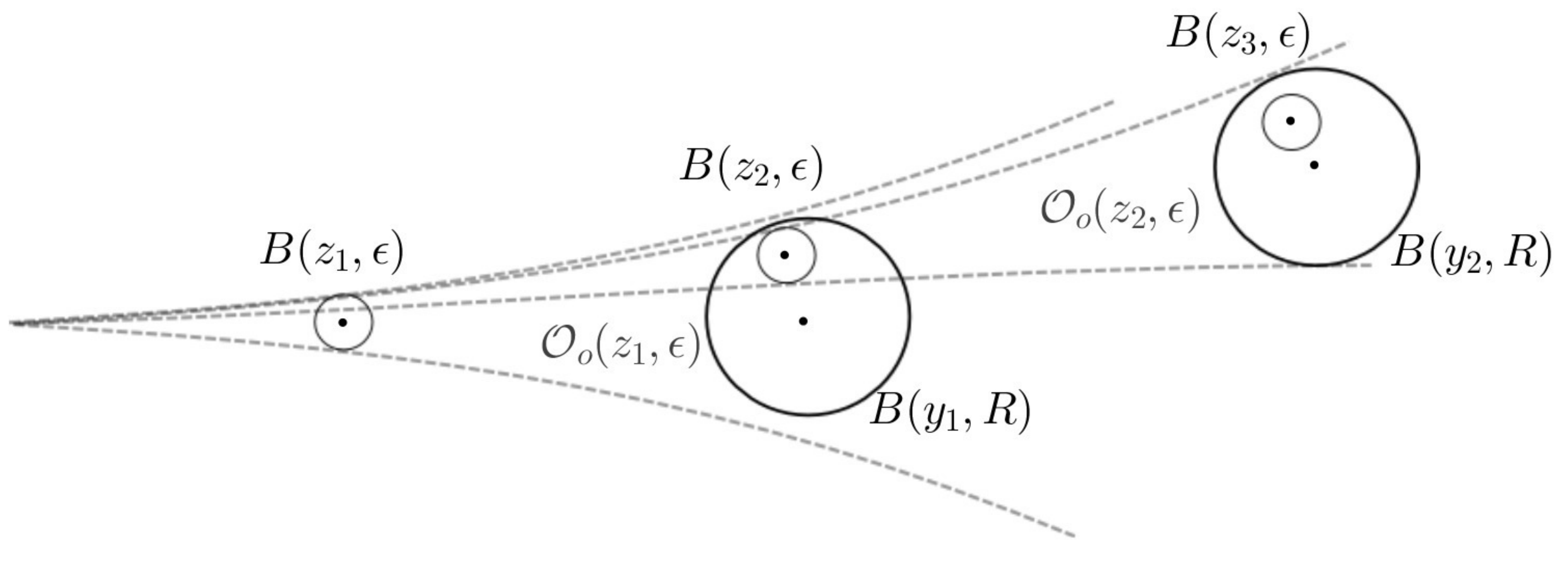}
\caption{}
\label{figure2}
\end{center}
\end{figure}

\noindent
By Arzel\`a-Ascoli theorem, up to subsequence, the sequence of segments $[z_1,z_n]$ converges to some geodesic ray $r$ uniformly on compact sets. By construction, $r$ meets each closed ball $\bar{B}(z_n,\epsilon)$ at some point $r(t_n)$ such that $\epsilon\leq t_{n+1}-t_n\leq R+K+\epsilon$ and $B(r(t_n),\epsilon)\subset B(z_n,2\epsilon)$ is a CAT(-1) neighborhood.
\end{proof}

\noindent
Now, we focus on the second step of the proof of Theorem \ref{thm:main}. Namely:

\begin{prop}\label{prop2}
Let $X$ be a CAT(0) space and $\gamma : \mathbb{R} \to X$ a geodesic line passing through infinitely many CAT(-1) points in such a way that the distance between any two consecutive such points is uniformly bounded from above and from zero and that the radii of the CAT(-1) neighborhoods are uniformly bounded away from zero. Then $\gamma$ is a contracting geodesic. 
\end{prop}

\noindent
Inspired from \cite{KapovichLeebCones}, our proof is based on the Gauss-Bonnet formula. We begin by recalling a few definitions.

\begin{definition}
Let $X$ be a polygonal complex. A \emph{corner} $(v,C)$ is the data of a vertex $v \in X$ and a polygon $C \subset X$ containing $v$. Given a vertex $v \in X$ (resp. a polygon $R \subset X$), we denote by $\mathrm{corner}(v)$ (resp. $\mathrm{corner}(R)$) the set of corners based at $v$ (resp. the set of corners supported by $R$). 
\end{definition}

\begin{definition}
An \emph{angled polygonal complex} $(X,\angle)$ is the data of a polygonal complex $X$ and a map
$$\angle : \{ \text{corners of $X$} \} \to \mathbb{R}.$$
The \emph{curvature of a vertex} $v \in X$ is defined as
$$\kappa(v)= 2\pi - \pi \cdot \chi \left( \mathrm{link}(v) \right) - \sum\limits_{c \in \mathrm{corner}(v)} \angle(v),$$
and the \emph{curvature of a polygon} $R \subset X$ as
$$\kappa(R)= \sum\limits_{c \in \mathrm{corner}(R)} \angle (c) - \pi \cdot | \partial R | +2 \pi.$$
It is worth noticing that, if $R$ is a triangle, then its curvature coincides with minus its \emph{deficiency}, denoted by $\mathrm{def}(R)$, ie., the difference between $\pi$ and the sum of its angles. 
\end{definition}

\noindent
As proved in \cite[Theorem 4.6]{McCammondWise}, this formalism allows us to recover a combinatorial version of the well-known Gauss-Bonnet formula. 

\medskip \noindent
\textbf{Combinatorial Gauss-Bonnet formula.} Let $(X,\angle)$ be an angled polygonal complex. Then
$$\sum\limits_{\text{$v$ vertex}} \kappa(v)+ \sum\limits_{\text{$R$ polygon}} \kappa(R) = 2 \pi \cdot \chi(X).$$

\medskip \noindent
Now we are ready to prove Proposition \ref{prop2}. 

\begin{proof}[Proof of Proposition \ref{prop2}.]
Let $x \notin \gamma$ and $y$ be two points such that $d(x,y)<d(x,\gamma)$, and let $z,w \in \gamma$ denote the projections of $x$ and $y$ onto $\gamma$ respectively. Fix some constants $\epsilon, L,R>0$ such that a ball of radius $\epsilon$ centered at some CAT(-1) point is CAT(-1) and such that the distance between any two consecutive CAT(-1) points along $\gamma$ is at most $L$ and at least $R$. Without loss of generality, we may suppose that $R > 2 \epsilon$. Suppose that there exist $N$ consecutive CAT(-1) points $x_1, \ldots, x_N \in \gamma$ between $z$ and $w$ such that $B(x_i,\epsilon) \cap [x,y] = \emptyset$ for every $1 \leq i \leq N$. For every $1 \leq i \leq N$, fix a point $y_i \in [z,x] \cup [x,y]\cup [y,w]$ whose projection onto $\gamma$ is $x_i$; notice that, because the projection of $[x,z]$ is reduced to the singleton $\{z \}$ and the projection of $[w,y]$ is reduced to the singleton $\{w \}$, necessarily $y_i \in [x,y]$. Finally, for every $1 \leq i \leq N$, let $a_i \in [x_{i-1},x_i]$, $b_i \in [x_i,y_i]$ and $c_i \in [x_{i},x_{i+1}]$ be the unique points of the corresponding segments at distance $\epsilon$ from $x_i$ (for convenience, we set $x_0=z$ and $x_{N+1}=w$). The configuration is summarised by Figure \ref{figure1}. 	
\begin{figure}
\begin{center}
\includegraphics[trim={0 0 11cm 13cm},clip,scale=0.71]{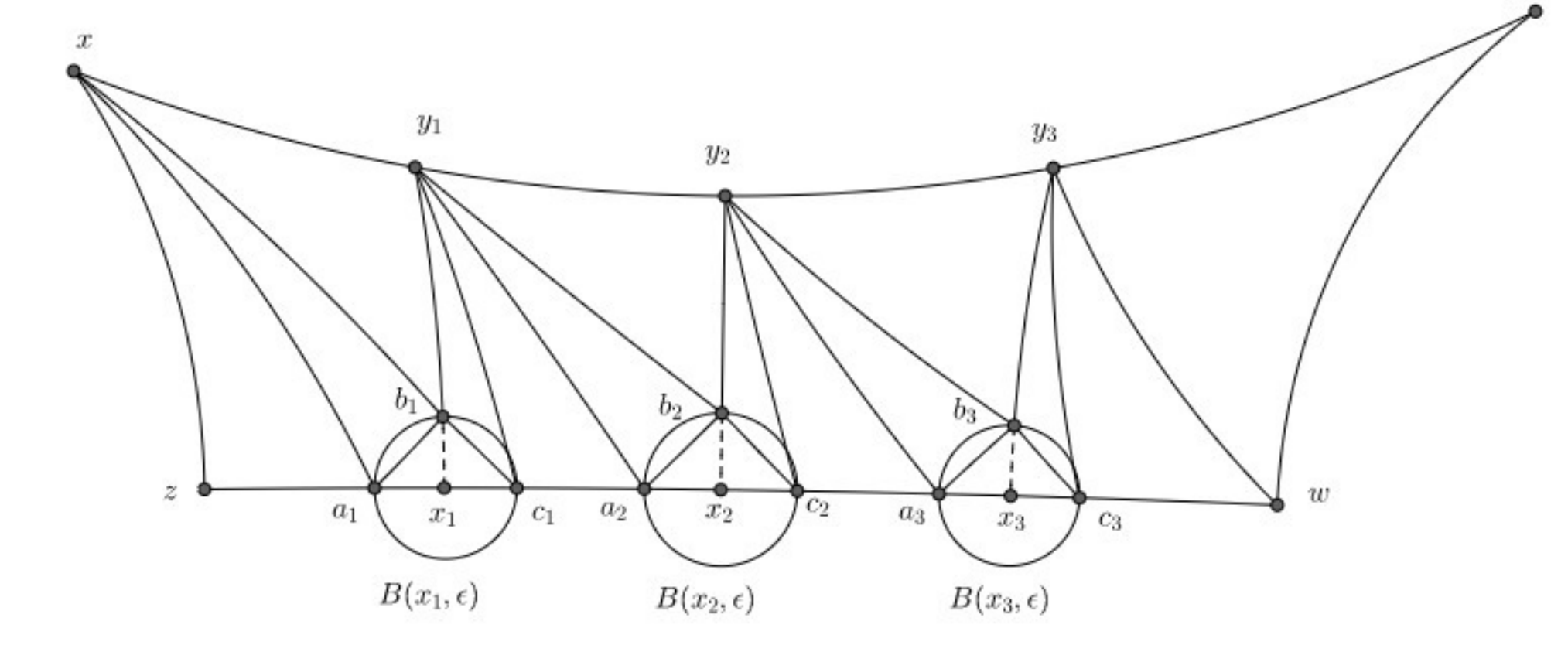}
\caption{}
\label{figure1}
\end{center}
\end{figure}

\medskip \noindent
Denote by $\Delta$ the union of the geodesics 

\begin{itemize}
	\item $[x,y]$, $[y,w]$, $[w,z]$ and $[z,x]$;
	\item $[y_i,b_i]$ for every $1 \leq i \leq N$;
	\item $[b_i,a_i]$ and $[b_i,c_i]$ for every $1 \leq i \leq N$;
	\item $[y_i,c_i]$, $[y_i,a_{i+1}]$ and $[y_i,b_{i+1}]$ for every $0 \leq i \leq N-1$ (setting $y_0=x$ and $x_0=z$).
\end{itemize}
Roughly speaking, $\Delta$ is a triangulation of the geodesic quadrilateral $Q(x,y,w,z)$. Now, consider a comparison quadrilateral $\overline{Q}(\bar{x},\bar{y}, \bar{w}, \bar{z})$ of $Q(x,y,w,z)$. For every $1 \leq i \leq N$, let $\overline{x}_i, \overline{y}_i, \overline{a}_i, \overline{c}_i$ denote the preimages of $x_i,y_i,a_i,c_i$ respectively under the comparison map $\overline{Q}(\bar{x},\bar{y}, \bar{w}, \bar{z}) \to Q(x,y,w,z)$. This map has a natural extension 
$$\overline{Q}(\bar{x},\bar{y}, \bar{w}, \bar{z}) \cup \bigcup\limits_{i=1}^N [\overline{y}_i, \overline{x}_i] \to Q(x,y,w,z) \cup \bigcup\limits_{i=1}^N [y_i,x_i],$$
sending each segment $[\overline{y}_i, \overline{x}_i]$ to the geodesic $[y_i,x_i]$ by an affine map. We denote by $\overline{b}_i$ the preimage of $b_i$ under this map, for every $1 \leq i \leq N$. By triangulating $\overline{Q}(\bar{x},\bar{y}, \bar{w}, \bar{z})$ as $Q(x,y,w,z)$, one gets a planar triangle complex $\overline{\Delta}$, and a map $f : \overline{\Delta} \to \Delta$ extending the comparison map $\overline{Q}(\bar{x},\bar{y}, \bar{w}, \bar{z}) \to Q(x,y,w,z)$. 

\medskip \noindent
From now on, the triangles $(\overline{c}_i,\overline{a}_i,\overline{b}_i)$ for $1 \leq i \leq N$, and their images under $f$, will be referred to as CAT(-1) triangles. Notice that a CAT(-1) triangle in $\Delta$ is contained into a ball of radius $\epsilon$ centered at a CAT(-1) point, so it lies in a CAT(-1) subspace. For every triangle $\delta$ of $\overline{\Delta}$, we assign angles to its corners by the following rules:
\begin{itemize}
	\item for non CAT(-1) triangles, we assign the angles of a comparison triangle in $\mathbb{E}^2$;
	\item for the CAT(-1) triangle $(\bar{a}_i,\bar{b}_i,\bar{c}_i)$ we assign to the vertex $\bar{a}_i$ (respectively $\bar{c}_i$) the corresponding angle in a comparison triangle in $\mathbb{H}^2$ for $(a_i,b_i,x_i)$ (respectively $(c_i,b_i,x_i)$) and to the vertex $\bar{b}_i$ the corresponding angle in a comparison triangle in $\mathbb{H}^2$ for $(a_i,b_i,c_i)$.
\end{itemize}
We emphasize that the angle $\angle_{\bar{c}_i} \left( \bar{a}_i, \bar{b}_i \right)$, for instance, denotes the angle as defined above, and not the corresponding angle in the Euclidean triangulation $\overline{\Delta}$.

\medskip \noindent
The plan is to apply the Gauss-Bonnet formula in order to bound the number $N$. So we need to investigate the curvatures of the vertices and triangles of $\overline{\Delta}$. 

\begin{claim}
The curvature of any vertex of $\overline{\Delta}$, different from $\bar{x},\bar{y},\bar{z}$ or the $\bar{b}_i$ for $1 \leq i \leq N$, is nonpositive.
\end{claim}

\noindent
Fix some $1 \leq i \leq N$. We want to compute the curvature of $\bar{c}_i$. Notice that
$$\angle_{\overline{c}_i}( \overline{a}_i, \overline{b}_{i}) + \angle_{\overline{c}_i}( \overline{b}_{i}, \overline{y}_i) + \angle_{\overline{c}_i}( \overline{y}_i, \overline{a}_{i+1})$$
is at least
$$\angle_{c_i}( a_i, b_{i}) + \angle_{c_i}( b_{i}, y_i) + \angle_{c_i}( y_i, a_{i+1}) \geq \angle_{c_i}(a_i,a_{i+1})= \pi.$$
Consequently, $\kappa(\overline{c_i}) \leq 0$. The same argument implies that $\kappa(\bar{y}_i)$ and $\kappa(\bar{a}_i)$ are nonpositive. This concludes the proof of our claim.

\medskip \noindent
Let $(a_i',b_i',x_i')$ and $(c_i'',b_i'',x_i'')$ be comparison triangles in $\mathbb{H}^2$ of $(a_i,b_i,x_i)$ and $(c_i,b_i,x_i)$ respectively.  These are isosceles triangles with angles at least $\frac{\pi}{2}$ at $x_i'$ and $x_i''$ respectively. In particular, by definition of the angles in $\bar{\Delta}$,
$$\angle_{\bar{a}_i}(\bar{b}_i,\bar{c}_i)=\angle_{a_i'}(b_i',x_i')=\angle_{b_i'}(a_i',x_i'),$$
$$\angle_{\bar{c}_i}(\bar{b}_i,\bar{a}_i)=\angle_{c_i''}(b_i'',x_i'')=\angle_{b_i''}(c_i'',x_i''),$$
where $\angle_{a_i'}(b_i',x_i')$, $\angle_{b_i'}(a_i',x_i')$, $\angle_{c_i''}(b_i'',x_i'')$, $\angle_{b_i''}(c_i'',x_i'')$ are the angles measured in $\mathbb{H}^2$.
\begin{claim}\label{claim2}
For $1 \leq i \leq N$, we have $\kappa(\bar{b}_i)\leq\angle_{b_i'}(a_i',x_i')+\angle_{b_i''}(x_i'',c_i'')-\angle_{\bar{b}_i}(\bar{a}_i,\bar{c}_i)$.
\end{claim}

\noindent
Fix some $1 \leq i \leq N$. The same argument as before shows that 
$$u:=\angle_{\overline{b}_i}( \overline{y}_i, \overline{y}_{i-1}) + \angle_{\overline{b}_i}( \overline{y}_{i-1}, \overline{a}_i) + \angle_{b_i'}(a_i', x_i')+\angle_{\overline{b}_i}( \overline{y}_{i}, \overline{c}_i) + \angle_{b_i''}(c_i'', x_i'')$$
is at least $2\pi$. Since
$$\kappa(\bar{b}_i)=2\pi-(u-\angle_{b_i'}(a_i',x_i')-\angle_{b_i''}(x_i'',c_i'')+\angle_{\bar{b}_i}(\bar{a}_i,\bar{c}_i)),$$
this proves the claim.

\medskip \noindent
Next, notice that the curvature of a triangle of $\overline{\Delta}$ which is not CAT(-1) is zero, since its angles come from a Euclidean triangle. Therefore,
$$\sum\limits_{\text{$\delta$ triangle}} \kappa(\delta) = \sum\limits_{\text{$\delta$ CAT(-1)}} \kappa(\delta).$$

\noindent
Let $\delta_i$ be the triangle $(\bar{a}_i,\bar{b}_i,\bar{c}_i)$ in $\bar{\Delta}$. Recall that the \emph{deficiency} $\mathrm{def}(\delta_i)$ of $\delta_i$ is the difference between $\pi$ and the sum of its angles, that is
\begin{align*}
\mathrm{def}(\delta_i)&=\pi-(\angle_{\bar{a}_i}(\bar{b}_i,\bar{c}_i)+\angle_{\bar{c}_i}(\bar{a}_i,\bar{b}_i)+\angle_{\bar{b}_i}(\bar{a}_i,\bar{c}_i))\\
&=\pi-(\angle_{a_i'}(b_i',x_i')+\angle_{c_i''}(b_i'',x_i'') +\angle_{\bar{b}_i}(\bar{a}_i,\bar{c}_i))\\
&=\pi-(\angle_{b_i'}(a_i',x_i')+\angle_{b_i''}(c_i'',x_i'') +\angle_{\bar{b}_i}(\bar{a}_i,\bar{c}_i)).
\end{align*}

\begin{claim}
$\displaystyle \sum\limits_{i=1}^N(\kappa(\bar{b}_i)-\mathrm{def}(\delta_i))\leq  -N\left(\pi-4\arccos \left( \sqrt{\frac{\cosh(\epsilon)}{\cosh(\epsilon)+1}} \right)\right)$.
\end{claim}

\noindent
We will use the following hyperbolic trigonometry formula:
\begin{lemma}\label{trigo}
\emph{\cite[Theorem 2.2.1(ii)]{BuserGeometry}}
Let $\Delta$ be a geodesic triangle in $\mathbb{H}^2$ with sides of length $a,b,c$ and respective opposite angles $\alpha,\beta,\gamma$. Then
$$\cos(\gamma)=\sin(\alpha)\sin(\beta)\cosh(c)-\cos(\alpha)\cos(\beta).$$
\end{lemma}
\noindent
As a consequence of Claim \ref{claim2},
\begin{align*}
\kappa(\bar{b}_i)-\mathrm{def}(\delta_i)&\leq \angle_{b_i'}(a_i',x_i')+\angle_{b_i''}(x_i'',c_i'')-\angle_{\bar{b}_i}(\bar{a}_i,\bar{c}_i) +\angle_{b_i'}(a_i',x_i')+\angle_{b_i''}(c_i'',x_i'') \\ & \hspace{0.5cm} +\angle_{\bar{b}_i}(\bar{a}_i,\bar{c}_i)-\pi\\
&\leq 2(\angle_{b_i'}(a_i',x_i')+\angle_{b_i''}(x_i'',c_i''))-\pi.
\end{align*}
\noindent
By applying Lemma \ref{trigo} to the triangle $(a_i',b_i',x_i')$, we deduce that 
$$\begin{array}{lcl} 0 & \geq & \cos \angle_{x_i'}(a_i',b_i') \\ \\ & \geq & \sin \angle_{a_i'}(b_i',x_i') \sin \angle_{b_i'}(a_i',x_i') \cosh( d(a_i',b_i')) - \cos \angle_{a_i'}(b_i',x_i') \cos \angle_{b_i'}(a_i',x_i') \\ \\ & \geq & \left( 1- \cos^2 \angle_{b_i'}(a_i',x_i') \right) \cosh(d(a_i,b_i)) - \cos^2 \angle_{b_i'}(a_i',x_i') \end{array}$$
So we get, since $\angle_{b_i'}(a_i',x_i')\leq\frac{\pi}{2}$, that
$$\cos\angle_{b_i'}(a_i',x_i')\geq \sqrt{\frac{\cosh(d(a_i,b_i))}{1+\cosh(d(a_i,b_i))}}.$$
Notice that $d(a_i,b_i)$ is greater than $\epsilon$ since the triangle $(a_i',b_i',x_i')$ is isoscele with angle at least $\frac{\pi}{2}$ at $x_i'$ and $d(x_i',a_i')=\epsilon$. Also notice that $x\mapsto\frac{\cosh(x)}{\cosh(x)+1} $ is increasing. Therefore, since cosinus is a decreasing function on $[0,\pi]$, we obtain that
$$\angle_{b_i'}(a_i',x_i')\leq \arccos\left(\sqrt{\frac{\cosh(\epsilon)}{1+\cosh(\epsilon)}}\right).$$
Similarly, one has 
$$\angle_{b_i''}(c_i'',x_i'')\leq \arccos\left(\sqrt{\frac{\cosh(\epsilon)}{1+\cosh(\epsilon)}}\right).$$
Thus,
$$\kappa(\bar{b}_i)-\mathrm{def}(\delta_i)\leq 4\arccos\left(\sqrt{\frac{\cosh(\epsilon)}{1+\cosh(\epsilon)}}\right)-\pi.$$
This concludes the proof of the claim.

\medskip \noindent
Finally, by applying the Gauss-Bonnet formula to $\overline{\Delta}$, one deduces that
$$2\pi=\sum\limits_{\text{$v$ vertex}} \kappa(v)+ \sum\limits_{\text{$\delta$ triangle}} \kappa(\delta)\leq 4\pi-N\left(\pi-4\arccos \left( \sqrt{\frac{\cosh(\epsilon)}{\cosh(\epsilon)+1}} \right) \right),$$
hence
$$N\left(\pi-4\arccos \left( \sqrt{\frac{\cosh(\epsilon)}{\cosh(\epsilon)+1}} \right)\right)\leq 2 \pi.$$
Notice that $\pi-4\arccos \left( \sqrt{\frac{\cosh(\epsilon)}{\cosh(\epsilon)+1}} \right)$ is strictly positive (or equivalently, $\frac{\cosh(\epsilon)}{\cosh(\epsilon)+1}$ is strictly greater than $\frac{1}{2}$) since $x\mapsto\frac{\cosh(x)}{\cosh(x)+1}$ is increasing and takes value $\frac 12$ at zero. Hence,
$$N\leq\eta(\epsilon):=\frac{2 \pi}{\pi-4\arccos \left( \sqrt{\frac{\cosh(\epsilon)}{\cosh(\epsilon)+1}} \right) }.$$

\noindent
So far, we have proved that, given our points $x \notin \gamma$ and $y$ such that $d(x,y)<d(x,\gamma)$, there exist at most $\eta(\epsilon)$ CAT(-1) points along $\gamma$ separating the projection $w$ of $y$ and the projection $z$ of $x$ onto $\gamma$ such that the balls of radii $\epsilon$ centered at these points are all disjoint from the geodesic $[x,y]$. Therefore, either $[z,w]$ contains at most $\eta(\epsilon)$ consecutive CAT(-1) points, so that
$$d(z,w)  \leq( \eta(\epsilon) +1) \cdot L,$$
or there exist more than $\eta(\epsilon)$ consecutive CAT(-1) points in $[z,w]$. In this case, let $c \in [z,w]$ denote the $(\eta(\epsilon)+1)$-th CAT(-1) point (starting from $z$ to $w$); then $[x,y]$ must intersect the ball $B(c,\epsilon)$ at some point $c'$. So, we get
$$d(x,c')+d(c',y)=d(x,y) < d(x,z)\leq d(x,c')+d(c',c)+d(c,z),$$
where the first inequality comes from the fact that $d(x,z)=d(x,\gamma)$. Thus
\begin{align*}
	d(c',y)&\leq d(c',c)+d(c,z) \\
	&\leq \epsilon+(\eta(\epsilon)+2)\cdot L.
\end{align*}
Finally, if $c''$ denotes the projection of $c'$ onto $\gamma$, since the projection onto $\gamma$ is 1-Lipschitz, we have
\begin{align*}
	d(z,w)&\leq d(z,c)+d(c,c'')+d(c'',w) \\
	&\leq (\eta(\epsilon)+2)\cdot L+d(c,c'')+d(c,y)\\
	&\leq 2(\eta(\epsilon)+2)\cdot L+3\epsilon.
\end{align*}
Consequently, the geodesic $\gamma$ is $B$-contracting where $B= 2(\eta(\epsilon)+2)\cdot L+3\epsilon $.
\end{proof}

\begin{remark}
As it was pointed out to us by Alexander Lytchak, a more geometric proof of Proposition \ref{prop2} is possible. Indeed, following the same lines, the combinatorial Gauss-Bonnet formula can be replaced with a geometric analogue proved by Alexandrov. We refer to \cite[Chapter V.3]{Alexandrov} for more information. 
\end{remark}

\begin{proof}[Proof of Theorem \ref{thm:main}.]
According to Proposition \ref{prop1}, there exists a geodesic ray $r$ and an increasing sequence $(t_n)$ of nonnegative reals tending to $+ \infty$ such that the point $r(t_n)$ has a CAT(-1) neighborhood with radius uniformly bounded away from zero for every $n \geq 0$ and such that the sequence $(t_{n+1}-t_n)$ is bounded. For every $n \geq 0$, let $g_n \in G$ be an isometry translating $r(t_n)$ into some fixed compact fundamental domain. Because $X$ is proper, we deduce from Arzel\`a-Ascoli theorem that the sequence of rays $(g_n \cdot r_n)$ subconverges to a geodesic line $\gamma$. Moreover, since the limit of a sequence of CAT(-1) points with CAT(-1) neighborhood of radii uniformly bounded away from zero must be a CAT(-1) point according to Fact \ref{fact}, we know that $\gamma$ passes through infinitely many CAT(-1) points in such a way that the distance between any two consecutive such points is uniformly bounded from above and from zero. According to Proposition \ref{prop2}, this geodesic must be contracting. The desired conclusion follows from Theorem \ref{thm:whenacylhyp}.
\end{proof}

\noindent
Let us conclude this section with an open question and a few examples. Define \emph{partially CAT(-1) groups} as groups acting geometrically on geodesically complete proper CAT(0) spaces containing points with CAT(-1) neighborhoods. It would be interesting to determine if partially CAT(-1) groups satisfy stronger hyperbolic properties than just being acylindrically hyperbolic. Notice that, in Section \ref{section:nonRH}, an example of a partially CAT(-1) group which is not relatively hyperbolic is given. A naive question is:

\begin{question}
Must an acylindrically hyperbolic CAT(0) group be partially CAT(-1)? For instance, when is a right-angled Artin group partially CAT(-1)?
\end{question}

\noindent
Now, let us indicate how to construct examples of partially CAT(-1) groups by gluing surfaces together.

\begin{ex}
Let $\Sigma_g$ be a compact orientable surface of genus $g \geq 2$. We think of $\Sigma_g$ as obtained from a hyperbolic regular right-angled $4g$-gon by identifying parallel sides in the usual way. Each side of this polygon defines a loop in $\Sigma_g$, which we refer to as a \emph{canonical loop}. Similarly, we think of the torus $\Sigma_1$ as obtained from a Euclidean square by identifying parallel sides in the usual way; and we refer to the image in $\Sigma_1$ of any side of our square as a \emph{canonical loop}. Up to rescaling our metrics, we may suppose without loss of generality that any canonical loop in any of our surfaces has a fixed length. Now, let $S$ be a space obtained by gluing compact orientable surfaces along canonical loops by isometries; let $\Sigma$ denote the underlying collection of surfaces. The universal cover $\widetilde{S}$ of $S$ is naturally a polygonal complex, 
which turns out to be CAT(0) provided that the graph underlying our graph of spaces is triangle-free. 
It follows from Theorem \ref{thm:main} that the fundamental group $\pi_1(S)$ must be acylindrically hyperbolic if the collection $\Sigma$ contains at least one surface of genus at least two. 
\end{ex}

\section{Concluding remarks}\label{section:nonRH}

\begin{figure}
\begin{center}
\includegraphics[trim={0 0 0 0},clip,scale=0.35]{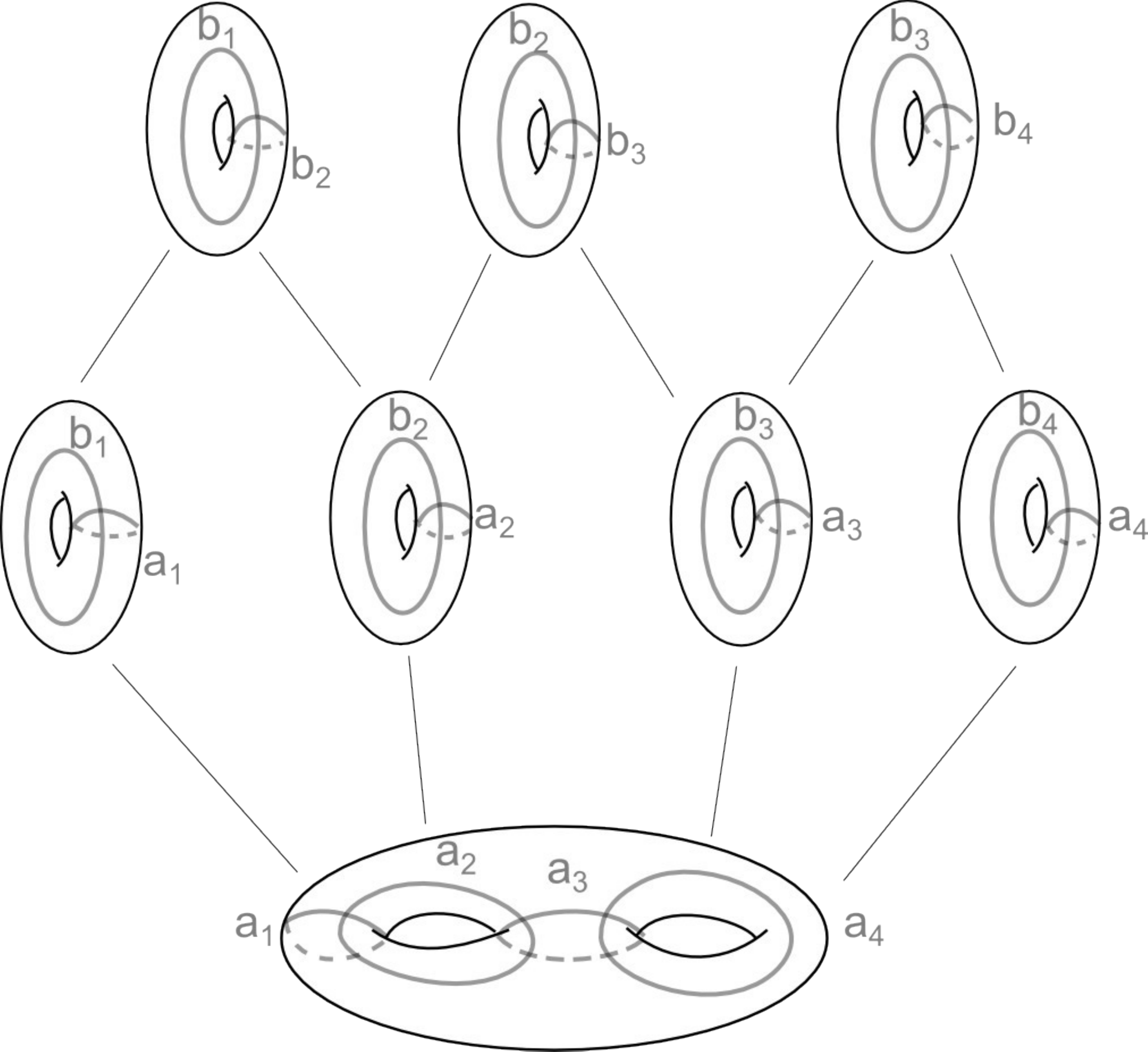}
\caption{A non relatively hyperbolic partially CAT(-1) group.}
\label{figure3}
\end{center}
\end{figure}
\noindent
As a conclusion of the article, we would to end with two remarks. The first one is that there exist plenty of partially CAT(-1) groups which are not relatively hyperbolic, showing that the conclusion of Theorem \ref{thm:main} cannot be strengthened in this direction. Let us construct an explicit example.

\medskip \noindent
Let $\Sigma$ denote the graph of spaces given by Figure \ref{figure3}. More explicitely, let $S$ be a compact orientable surface of genus two and $T_1,T_2,T_3,T_4, T_{12}, T_{23}, T_{34}$ seven tori. We denote by $a_1,a_2,a_3,a_4$ the canonical loops of $S$ (ordered by following a fixed cyclic order on the boundary of the polygon corresponding to $S$); by $a_i^i,b_i^i$ the canonical loops of the torus $T_i$; and by $b_i^{i,i+1}, b_{i+1}^{i,i+1}$ the canonical loops of the torus $T_{i,i+1}$. The space $\Sigma$ is obtained from these eight surfaces by gluing each $T_i$ to $S$ by identifying $a_i$ and $a_i^i$, and by gluing $T_{i,i+1}$ to $T_i$ and $T_{i+1}$ by identifying $b^{i,i+1}_i$ and $b^{i,i+1}_{i+1}$ respectively to $b_i^i$ and $b_{i+1}^i$. 

\medskip \noindent
 It follows from \cite[Proposition II.11.6]{MR1744486} that $\Sigma$ is locally CAT(0). Moreover, each subsurface is locally isometrically embedded, so that any point on the surface of genus two which does not belong to any torus admits a CAT(-1) neighborhood. Next, notice that the universal cover $\widetilde{\Sigma}$ of $\Sigma$ is geodesically complete, as a union of Euclidean and hyperbolic planes (which are convex in $\widetilde{\Sigma}$). Consequently, the fundamental group $G$ of $\Sigma$ is partially CAT(-1). 

\begin{fact}
The group $G$ is not relatively hyperbolic. 
\end{fact}

\medskip \noindent
First of all, notice that $G$ admits the following presentation:
$$\left\langle \begin{array}{c} a_1,a_2,a_3,a_4 \\ b_1,b_2,b_3,b_4 \end{array} \left| \begin{array}{c} [a_1,a_2] [a_3,a_4]= [b_1,b_2]= [b_2, b_3]=[b_3,b_4]=1 \\ \text{$[ a_1,b_1 ] =[a_2,b_2]= [a_3,b_3]= [a_4,b_4]=1$} \end{array} \right.  \right\rangle$$
Suppose that our group $G$ is hyperbolic relatively to some finite collection of subgroups. As a consequence of \cite[Theorems 4.16 and 4.19]{OsinRelativeHyp}, every non cyclic free abelian subgroup of $G$ must be contained into some peripheral subgroup, so, for every $1 \leq i \leq 4$ (resp. $1 \leq j \leq 3$), there exists some peripheral subgroup $H_i$ (resp. $K_j$) containing both $a_i$ and $b_i$ (resp. $b_j$ and $b_{j+1}$). Notice that
$$b_i \in H_i \cap K_i \ \text{for every $1 \leq i \leq 4$, and} \ b_i \in K_i \cap K_{i+1} \ \text{for every $1 \leq i \leq 3$}.$$
Since the collection of peripheral subgroups must be malnormal according to \cite[Theorems 1.4 and 1.5]{OsinRelativeHyp}, it follows that $H_1= H_2=H_3= H_4= K_1=K_2=K_3$. Let $H$ denote this common peripheral subgroup. Since $H$ contains all the generators, necessarily $G=H$. This proves our fact. 

\medskip \noindent
Our second and last remark is the following. The rough idea of our main result is that a group which is nonpositively-curved everywhere and negatively-curved at at least one point must be acylindrically hyperbolic (or virtually cyclic). We expect that such a result holds in other contexts, by interpreting the expressions ``nonpositively-curved'' and ``negatively-curved'' in a different way. As an illustration, let us prove the following statement:

\begin{prop}\label{prop:cubical}
Let $G$ be a group acting essentially and geometrically on a CAT(0) cube complex $X$. Suppose that $X$ contains a \emph{CAT(-1) vertex}, ie., a vertex without induced cycles of length four in its link. Then $G$ is either acylindrically hyperbolic or virtually cyclic.
\end{prop}

\noindent
It is worth noticing that, endowed with its usual CAT(0) metric, a CAT(0) cube complex cannot contain a point admitting a CAT(-1) neighborhood, as each cube is isometric to a (flat) Euclidean cube. So Theorem \ref{thm:main} cannot apply. A naive attempt to get more hyperbolicity would be to identify each cube with a fixed hyperbolic cube, but the cube complex one obtains in that way may be no longer CAT(0). Nevertheless, in \cite[Section 4.2.C]{GromovHyp}, Gromov notices that, by endowing each cube of a given CAT(0) cube complex with the metric of a hyperbolic cube, the geodesic metric space thus obtained turns out to be CAT(-1) provided the vertices of our cube complex have no cycles of length four in their links. This observation justifies the terminology used in the statement of Proposition \ref{prop:cubical}. 

\begin{proof}[Proof of Proposition \ref{prop:cubical}.]
If $X$ contains a CAT(-1) vertex, then it cannot split as a Cartesian product of two cube complexes. Therefore, \cite[Theorem 6.3]{MR2827012} implies that $G$ contains a contracting isometry, say $g \in G$. The desired conclusion follows from Theorem~\ref{thm:whenacylhyp}. 
\end{proof}

\addcontentsline{toc}{section}{References}

\bibliographystyle{alpha}
\bibliography{BarelyCAT}

\Address

\end{document}